\documentclass[11pt]{article}

\usepackage[T1]{fontenc}
\usepackage{lmodern}
\usepackage{amsmath,amssymb,amsthm}
\usepackage{mathtools}
\usepackage{enumitem}
\usepackage{booktabs,longtable,array}
\usepackage{multicol}
\usepackage{pdflscape}
\usepackage{verbatim}
\usepackage{microtype}
\usepackage[hidelinks,hypertexnames=false]{hyperref}
\usepackage[margin=1in]{geometry}

\newtheorem{theorem}{Theorem}[section]
\newtheorem{lemma}[theorem]{Lemma}
\newtheorem{corollary}[theorem]{Corollary}
\newtheorem{proposition}[theorem]{Proposition}
\theoremstyle{remark}

\newtheorem{example}[theorem]{Example}

\newcommand{\Z}{\mathbb{Z}}
\newcommand{\Cfive}{C_5}
\newcommand{\ZenodoDOI}{\texttt{https://doi.org/10.5281/zenodo.21882725}}

\title{Completing the Boundary Case of the Mahmoodian--Mirzakhani Conjecture
and 117 New Computational 5-Cycle Decompositions of Complete Tripartite Graphs}
\author{Roozbeh Pournader}
\date{Preprint, August 17, 2026}

\begin{document}

\maketitle

\begin{abstract}
Let $K_{r,s,t}$, with $r\le s\le t$, denote the complete tripartite graph
whose partite sets have sizes $r,s,t$. Mahmoodian and Mirzakhani~\cite{MahmoodianMirzakhani1995} gave three
necessary conditions for $K_{r,s,t}$ to admit a decomposition into 5-cycles
and conjectured that these conditions are sufficient. One of the conditions is $t\le 4rs/(r+s)$.
We prove the conjecture for every odd triple on the extremal boundary $t = 4rs/(r+s)$.

The proof is constructive. After reducing an arbitrary odd boundary triple to
$(r,s,t)=(hga,hgb,hab)$, $a+b=4g$, we give an explicit cyclic decomposition of $K_{ga,gb,ab}$ and use the 
Mahmoodian and Mirzakhani scaling theorem to supply the common factor $h$. Together
with the previously known all-even result, this settles the conjecture for
every triple satisfying the boundary condition with equality.

We also report explicit computer-generated $C_5$-decompositions for 117 odd triples satisfying the necessary conditions, 116 of which are strict-interior cases.  To the best of our knowledge, all 117 cases were previously unresolved: no
decomposition for any of them had been reported, and none of the 117 triples
is covered by earlier existence results, constructions, or their
recursive consequences.  Moreover, these 117 certificates together with the boundary
construction settle every previously unresolved triple satisfying the
necessary conditions with fewer than $4400$ edges.
Each computation is supplied as a machine-readable cycle-list certificate and
can be checked independently by a short Python verifier.  We also give a complete
human-readable edge-label-matrix certificate for $K_{9,19,23}$.
\end{abstract}

\section{Introduction}

A $\Cfive$-decomposition of a graph $G$ is a partition of $E(G)$ into
cycles of length five. Mahmoodian and Mirzakhani~\cite{MahmoodianMirzakhani1995}
initiated the study of $\Cfive$-decompositions of complete tripartite graphs.
For positive integers $r\le s\le t$, they showed that if $K_{r,s,t}$ has
a $\Cfive$-decomposition, then
\begin{equation}
r,s,t \text{ have the same parity},
\tag{1}
\end{equation}
\begin{equation}
5\mid rs+rt+st,
\tag{2}
\end{equation}
and
\begin{equation}
t\le \frac{4rs}{r+s}.
\tag{3}
\end{equation}
They conjectured that these three conditions are also sufficient.

Mahmoodian and Mirzakhani already established two infinite families on the
boundary in~(3). Their Corollary~4.3 gives a $\Cfive$-decomposition of
$K_{2n,2n,4n}$ for every positive integer $n$, and their Corollary~4.4 gives
a $\Cfive$-decomposition of $K_{m,3m,3m}$ for every positive integer $m$
\cite[Corollaries~4.3 and~4.4]{MahmoodianMirzakhani1995}. These are exactly
the two possible equal-part forms on the boundary: if the two smaller parts
are equal, boundary equality gives $K_{r,r,2r}$, which is admissible only
when $r$ is even; if the two larger parts are equal, boundary equality gives
$K_{r,3r,3r}$. Thus the original paper settles the equal-part endpoints of
the boundary.

A number of further cases of the conjecture have subsequently been
established. Cavenagh and Billington~\cite{CavenaghBillington2000} treated
several additional families, and Cavenagh~\cite{Cavenagh2002} proved
sufficiency whenever $r,s,t$ are all even. Billington and
Cavenagh~\cite{BillingtonCavenagh2011} proved further results when odd
partite sets have similar sizes. Alipour, Mahmoodian and
Mollaahmadi~\cite{AlipourEtAl2012}, Abdolmaleki et
al.~\cite{AbdolmalekiEtAl2019}, and Kudarzi, Mahmoodian and
Naghdabadi~\cite{KudarziEtAl2021} obtained additional constructions for
portions of the remaining odd case.

In addition to the boundary construction, we give explicit computational
certificates for 117 triples.  In constructing this list, we compared every
candidate triple against the explicit examples, sufficient families, and
recursive constructions in the previous literature summarized above.  To the
best of our knowledge, no $C_5$-decomposition supported by a valid prior
construction had been reported for any of these cases, and none of these 117
existence statements is derivable from the previously available results and
constructions that we found.
Two of the triples,
$K_{15,21,25}$ and $K_{15,23,25}$, lie in the literal stated range of
a published mixing theorem, but the proof of that theorem requires
one-summand representations that do not exist for these instances; the
certificates supplied here settle them independently.
Taken together with the boundary construction
of Theorem~\ref{thm:main}, these certificates also close the entire previously
unresolved range below $4400$ edges: according to our reading of the cited
literature, every triple $r\le s\le t$ satisfying the necessary conditions
and $|E(K_{r,s,t})|<4400$ is either covered by an earlier result, by the
boundary theorem, or by one of the 117 certificates reported here.  The
computations are described in Section~\ref{sec:computation}.  Table~\ref{tab:computational-cases}
records all 117 triples, and Appendices~\ref{app:k9matrix} and~\ref{app:k9cycles} give a complete human-readable certificate for the first
example in the list, $K_{9,19,23}$. Machine-readable certificates and an
independent verifier have been deposited with the paper in the Zenodo record
\ZenodoDOI.

In this note we complete the extremal case of~(3),
\begin{equation}
t=\frac{4rs}{r+s},
\tag{4}
\end{equation}
for odd part sizes, including the case in which all three part sizes are
distinct.

Refining the edge-counting argument used by Mahmoodian and Mirzakhani in
the proof of their Theorem~2.1, let $x,y,z$ denote the numbers of cycles in
which the singleton vertex lies respectively in the parts of sizes $r,s,t$.
Every 5-cycle in a tripartite graph uses its vertices in a $(2,2,1)$
distribution among the three partite sets. Counting the edges between each
pair of parts gives
\[
3x+y+z=st,
\]
\[
x+3y+z=rt,
\]
\[
x+y+3z=rs.
\]
Consequently
\begin{equation}
x=\frac{4st-rs-rt}{10},\quad
y=\frac{4rt-rs-st}{10},\quad
z=\frac{4rs-rt-st}{10}.
\tag{5}
\end{equation}
Thus $z\ge0$ is exactly condition~(3), and equality in~(3) forces
\begin{equation}
z=0.
\tag{6}
\end{equation}

Our main result is the following.

\begin{theorem}\label{thm:main}
Let $r\le s\le t$ be positive odd integers satisfying
$t=4rs/(r+s)$. Then $K_{r,s,t}$ has a $\Cfive$-decomposition.
\end{theorem}

The proof occupies Sections~\ref{sec:param} and~\ref{sec:construction}.
The three partite sets are represented as pairwise products of cyclic
groups, and the edge sets are partitioned according to coordinate
differences.

\section{Parameterization of odd boundary triples}
\label{sec:param}

We first put the parameters in a convenient form.

\begin{lemma}\label{lem:param}
Suppose $r\le s\le t$ are positive odd integers that
satisfy $t=4rs/(r+s)$. Then there exist positive odd integers $h$, $g$, $a$, and $b$ such that
\begin{equation}
r=hga,\quad
s=hgb,\quad
t=hab,
\tag{7}
\end{equation}
where
\begin{equation}
\gcd(a,b)=1,\quad
a+b=4g,\quad
g\le a\le b\le3a.
\tag{8}
\end{equation}
\end{lemma}

\begin{proof}
Let $d=\gcd(r,s)$, $r=da$, and $s=db$,
where $\gcd(a,b)=1$ and $a\le b$. Since $r$ and $s$ are odd, so are
$d$, $a$, and~$b$. The boundary equality gives
\begin{equation}
t=\frac{4dab}{a+b}.
\tag{9}
\end{equation}
The numerator in~(9) has a $2$-adic valuation of exactly $2$. Since $t$ is an odd
integer, the denominator $a+b$ must also have a 2-adic valuation of exactly $2$.
Hence
\begin{equation}
a+b=4g
\tag{10}
\end{equation}
for some odd positive integer $g$.

We claim that
$\gcd(g,ab)=1$.
Indeed, if an odd prime $p$ divided both $g$ and $a$, then from
$4g=a+b$ it would also divide $b$, contradicting $\gcd(a,b)=1$. The
argument for $b$ is identical.

Equation~(9) now becomes
\[
t=\frac{dab}{g}.
\]
Since $t$ is an integer and $\gcd(g,ab)=1$, we must have $g\mid d$.
Write $d=hg$.
Then 
$r=hga$, $s=hgb$, and $t=hab$,
as required. Since $d$ and $g$ are odd, $h$ is also odd.

Finally, $s\le t$ gives $g\le a$. Since $a\le b$ and $a+b=4g$, $g\le a\le b$
and
$b=4g-a\le4a-a=3a$.
This proves the lemma.
\end{proof}

The case $b=3a$ in Lemma~\ref{lem:param} is precisely the previously known
odd equal-part endpoint. Indeed, $a+b=4g$ then gives $g=a$, and
$\gcd(a,b)=1$ forces $a=1$, so the primitive triple is $(1,3,3)$; the
common factor $h$ gives the family $K_{h,3h,3h}$ of
Mahmoodian and Mirzakhani.

For the remainder of the construction define
\begin{equation}
u=\frac{a-g}{2},\qquad
v=\frac{b-g}{2}.
\tag{11}
\end{equation}
Since $a$, $b$, and $g$ are odd and $g\le a\le b$, $u$ and $v$ are nonnegative
integers. The relation $a+b=4g$ yields the identities
\begin{equation}
u+v=g,
\tag{12}
\end{equation}
\begin{equation}
3u+v=a,
\tag{13}
\end{equation}
and
\begin{equation}
u+3v=b.
\tag{14}
\end{equation}
These identities are precisely what is needed for the construction.

\section{The cyclic construction}
\label{sec:construction}

We first deal with the case $h=1$, that is, $K_{ga,gb,ab}$. Let its three vertex classes be
$A=\Z_g\times\Z_a$,
$B=\Z_g\times\Z_b$,
and
$C=\Z_a\times\Z_b$. Write the corresponding vertices as
$A_{x,i}$, $B_{y,j}$, and~$C_{z,k}$. (All subscripts below are interpreted in their respective cyclic groups.)

Use the standard integer representatives for the cyclic groups and set
\begin{equation}
U=\{0,\ldots,u-1\},\qquad
V=\{u,\ldots,g-1\}.
\tag{15}
\end{equation}
Thus $|U|=u$ and $|V|=v$.  An interval is understood to be empty when its
upper endpoint is smaller than its lower endpoint.  The identities~(13) and~(14)
give the consecutive decompositions of difference classes
\begin{equation}
\Z_a=
\{0,\ldots,3u-1\} \cup
\{3u,\ldots,3u+v-1\},
\tag{16}
\end{equation}
and
\begin{equation}
\Z_b=
\{0,\ldots,u-1\} \cup
\{u,\ldots,u+3v-1\}.
\tag{17}
\end{equation}
We use these blocks directly in the cycle formulas.
For the second block of
$\Z_g$, it is convenient to write its elements as $u+p$, where
$0\le p<v$.

\subsection*{Family I}
For each
$0\le q<u$,
$x\in\Z_g$,
$i\in\Z_a$, and
$j\in\Z_b$, take
\begin{equation}
(
B_{x+q,j},
A_{x,i},
C_{i+3q,\,j},
A_{x,i-1},
C_{i+3q+1,\,j+q}
).
\tag{18}
\end{equation}
If this family is nonempty, then $a\ge3$.  Hence the two $A$-vertices are
distinct, and the two $C$-vertices are distinct because their first coordinates
differ by $1$.  Thus~(18) is a simple 5-cycle.  It has one $B$-vertex, two
$A$-vertices, and two $C$-vertices.

\subsection*{Family II}
For each
$0\le p<v$,
$x\in\Z_g$,
$i\in\Z_a$, and
$j\in\Z_b$, take
\begin{equation}
(
A_{x,i},
B_{x+u+p,j},
C_{i,\,j+u+3p},
B_{x+u+p,\,j-1},
C_{i+3u+p,\,j+u+3p+1}
).
\tag{19}
\end{equation}
If this family is nonempty, then $b\ge3$.  Hence the two $B$-vertices are
distinct, and the two $C$-vertices are distinct because their second coordinates
differ by $1$.  Thus~(19) is also a simple 5-cycle.  It has one $A$-vertex,
two $B$-vertices, and two $C$-vertices.

We next verify that these cycles form a decomposition.

\begin{lemma}\label{lem:cyclic}
The cycles in~(18) and~(19) form a $\Cfive$-decomposition of
$K_{ga,gb,ab}$.
\end{lemma}

\begin{proof}
For edges joining two partite sets, define the cyclic differences
\[
\Delta_{AB}(A_{x,i},B_{y,j})=y-x\in\Z_g,
\]
\[
\Delta_{AC}(A_{x,i},C_{z,k})=z-i\in\Z_a,
\]
\[
\Delta_{BC}(B_{y,j},C_{z,k})=k-j\in\Z_b.
\]
Reading the edge differences from~(18) and~(19) gives
\[
\begin{array}{c|c|c|c}
 & \Delta_{AB} & \Delta_{AC} & \Delta_{BC} \\
\hline
\text{Family I }(q)
 & q
 & 3q,\;3q+1,\;3q+2
 & q \\
\text{Family II }(p)
 & u+p
 & 3u+p
 & u+3p,\;u+3p+1,\;u+3p+2.
\end{array}
\tag{20}
\]
For a fixed $q$ or $p$ and a fixed edge position in one of the two families,
the free coordinates of that edge are obtained from $(x,i,j)$ only by adding
fixed constants in the relevant cyclic groups.  Consequently, as
$(x,i,j)$ ranges over $\Z_g\times\Z_a\times\Z_b$, that edge position runs
bijectively through all edges having the indicated difference.

It therefore remains only to check that the displayed differences partition
the three difference groups.  For $AB$-edges they are
\[
\{0,\ldots,u-1\} \cup
\{u,\ldots,u+v-1\}
=\{0,\ldots,g-1\}=\Z_g
\]
by~(12).  For $AC$-edges, Family I supplies the consecutive block
$\{0,\ldots,3u-1\}$ and Family II supplies
$\{3u,\ldots,3u+v-1\}$; by~(13), their disjoint union is $\Z_a$.
For $BC$-edges, Family I supplies $\{0,\ldots,u-1\}$ and Family II supplies
$\{u,\ldots,u+3v-1\}$; by~(14), their disjoint union is $\Z_b$.
Thus every edge of each of the three bipartite edge sets occurs exactly once,
which proves the lemma.
\end{proof}

As an independent numerical check, the construction contains
\[
u gab+v gab=gab(u+v)=g^2ab
\]
cycles. On the boundary,
\[
\begin{aligned}
|E(K_{ga,gb,ab})|
&=g^2ab+ga^2b+gab^2\\
&=gab(g+a+b)\\
&=5g^2ab,
\end{aligned}
\]
so the required number of 5-cycles is indeed $g^2ab$. The type counts are
$gabu$ cycles with singleton $B$, $gabv$ cycles with singleton $A$, and no
cycles with singleton $C$, in agreement with~(5) and~(6).

We now restore the common factor $h$ using a result from the original paper.
Mahmoodian and Mirzakhani proved that if $K_{r,s,t}$ admits a
$\Cfive$-decomposition, then so does $K_{ar,as,at}$ for every positive
integer $a$ \cite[Theorem~4.1]{MahmoodianMirzakhani1995}.

\begin{proof}[Proof of Theorem~\ref{thm:main}]
By Lemma~\ref{lem:param},
$(r,s,t)=(hga,hgb,hab)$
for odd $a$, $b$, and $g$ satisfying the hypotheses of Section~\ref{sec:construction}.
Lemma~\ref{lem:cyclic} supplies a $\Cfive$-decomposition of
$K_{ga,gb,ab}$. Applying the scaling theorem of Mahmoodian and Mirzakhani
\cite[Theorem~4.1]{MahmoodianMirzakhani1995} with multiplier $h$ gives a
$\Cfive$-decomposition of
$K_{hga,hgb,hab}=K_{r,s,t}$.
\end{proof}

\section{Consequences and an example}

The following is immediate.

\begin{corollary}\label{cor:odd-boundary}
The Mahmoodian--Mirzakhani necessary conditions are sufficient for every
odd triple $r\le s\le t$ satisfying
$t=4rs/(r+s)$.

\end{corollary}

Cavenagh~\cite{Cavenagh2002} proved the conjecture for all triples
$r,s,t$ that are even. Combining that result with
Corollary~\ref{cor:odd-boundary} gives the full boundary statement.

\begin{corollary}\label{cor:full-boundary}
Let $r\le s\le t$ have the same parity and satisfy
$t=4rs/(r+s)$. Then $K_{r,s,t}$ admits a $\Cfive$-decomposition.
\end{corollary}

Indeed, if the three parameters are even, this follows from
Cavenagh~\cite{Cavenagh2002}, while if they are odd it follows from
Theorem~\ref{thm:main}. Notice also that on the boundary
$rs+rt+st=rs+t(r+s)=5rs$,
so the divisibility condition is automatic.

\begin{example}[The graph $K_{35,65,91}$]
The triple
$(35,65,91)$
lies on the boundary since
\[
91=\frac{4\cdot35\cdot65}{35+65}.
\]
Here $a=7$, $b=13$, $g=5$,
so $u=(7-5)/2=1$ and
$v=(13-5)/2=4$. Family I has only $q=0$, and therefore consists of the cycles
\[
\bigl(B_{x,j},A_{x,i},C_{i,j},A_{x,i-1},C_{i+1,j}\bigr),
\]
where $x\in\Z_5$, $i\in\Z_7$, and $j\in\Z_{13}$.  Family II has
$p=0,1,2,3$ and consists of
\[
\bigl(
A_{x,i},
B_{x+1+p,j},
C_{i,j+1+3p},
B_{x+1+p,j-1},
C_{i+3+p,j+2+3p}
\bigr).
\]
Thus~(18) gives $455$ cycles with singleton $B$, while~(19) gives $1820$
cycles with singleton $A$.  Hence $K_{35,65,91}$ is decomposed into
$455+1820=2275$
5-cycles, as required from
$|E(K_{35,65,91})|=11375$.
The example illustrates that no search or auxiliary partition choices are
intrinsic to the boundary construction: once the standard cyclic labeling is
fixed, the cycles are given directly by~(18) and~(19).
\end{example}

\section{Computational decompositions}
\label{sec:computation}

The construction in Sections~\ref{sec:param}--\ref{sec:construction} settles
the complete boundary, but it does not apply in the strict interior
$t<4rs/(r+s)$,
where all three singleton cycle types may occur.  To make progress in that
region, and also to obtain explicit witnesses for selected boundary instances,
we carried out a direct computational search for decompositions of unresolved
odd triples satisfying the necessary conditions.  This section describes the
search representation, the certificate format, and the independent
verification procedure.  The resulting 117 certificates are summarized in
Table~\ref{tab:computational-cases}.

\subsection{Exact-cover formulation}
\label{subsec:exact-cover}

For a fixed triple $(r,s,t)$, label the partite sets
$A=\{A_0,\ldots,A_{r-1}\}$,
$B=\{B_0,\ldots,B_{s-1}\}$, and
$C=\{C_0,\ldots,C_{t-1}\}$. The computational problem is an exact-cover instance.  The columns are the
edges of $K_{r,s,t}$, and every simple 5-cycle is a candidate row containing
its five edges.  A set of rows is an exact cover precisely when the
corresponding cycles form a $C_5$-decomposition.

The implementation generated candidate cycles by singleton type.  Rather than
storing every cycle as a five-vertex object, rows were assigned compact
integer identifiers and decoded into their five vertices and five edge
columns when needed.  Incidence lists for edges were stored in packed form;
for the larger instances, row identifiers were local to the singleton type,
which reduced memory usage.  The search used a fail-first rule: at each step
it selected an uncovered edge having the fewest currently compatible
candidate cycles and branched over those cycles.  Candidate branches were
ordered by the current sizes of the five incident edge lists, with a
reproducible pseudorandom tie break.  A canonical cycle could be fixed at the
start to remove vertex-label symmetry.

The forced cycle-type counts from~(5) were also used.  Writing
$x_A,x_B,x_C$ for the numbers of cycles whose singleton vertex lies in
$A,B,C$, respectively,
\[
 x_A=\frac{4st-rs-rt}{10},\quad
 x_B=\frac{4rt-rs-st}{10},\quad
 x_C=\frac{4rs-rt-st}{10}.
\]
For instances close to the boundary, $x_C$ can be very small.  In those cases
an unrestricted exact-cover search may spend substantial effort on branches
that use the wrong global number of $C$-singleton cycles.  For several such
instances we therefore fixed the required $C$-singleton cycles first and
solved the residual exact-cover instance using only $A$- and $B$-singleton
rows.  A reversible-undo implementation allowed backtracking without
reconstructing the entire candidate-row universe.  The output of every
successful search was a plain-text cycle list, one 5-cycle per line.

The computation is used here only to produce witnesses.  The correctness of a
reported decomposition does not depend on trusting the search heuristic or
its implementation: each output file is independently checkable as described
next.

\subsection{Independent verification}
\label{subsec:verification}

The archive contains a short Python program, \texttt{verify\_c5\_tripartite.py}.
It is invoked as
\begin{center}
\texttt{python3 verify\_c5\_tripartite.py k$r$\_$s$\_$t$\_decomposition.txt}.
\end{center}
Each nonempty line of the certificate must contain five vertex labels.  The
verifier checks that the five vertices are distinct, that every label belongs
to the stated partite set, and that consecutive vertices lie in different
parts.  It then forms the five undirected edges of every listed cycle.  If
$E=rs+rt+st$,
the verifier requires exactly $E/5$ cycles and requires the set of edges
appearing in the file to have cardinality of exactly $E$.  Since the file then
contains exactly $E$ edge occurrences in total, cardinality $E$ of their
union simultaneously proves that no edge is repeated and that every edge of
$K_{r,s,t}$ occurs.  Thus acceptance by this verifier is a direct check that
the file is a $C_5$-decomposition certificate.

The verifier is intentionally independent of the search data structures: it
reads only the final text certificate.  Appendix~\ref{app:verifier} reproduces
the verifier used for the certificates in Table~\ref{tab:computational-cases}.

\subsection{The 117 verified cases}
\label{subsec:117cases}

For each triple in Table~\ref{tab:computational-cases}, a complete cycle-list
certificate was produced and accepted by the independent verifier.  Before a
triple was included in the table, we checked it against the explicit
examples, infinite families, sufficient conditions, and recursive
constructions in the previous literature cited in the Introduction.  To the
best of our knowledge, no entry had a previously reported decomposition or a
valid covering construction. (As noted in the Introduction, the literal
statement of one published mixing theorem includes some of these triples, but
its proof does not supply the required summand representations for them.  The
certificates here do not rely on that assertion.  Specifically, for
$K_{15,21,25}$ and $K_{15,23,25}$ the stated inequalities force the
one-summand case, whereas the proof requires the non-$5$-multiple part to
belong to $\{11,13,15,17,19,25\}$; neither $21$ nor $23$ belongs to this set.
This is necessarily a claim about the available literature rather than a
formal non-derivability theorem. It should also be noted that while 26 of our certificates satisfy $t \le (5/3)r$ and may appear covered by~\cite{AlipourEtAl2012} at first glance, that publication requires a lower bound of $75 \le r$, which is larger than all our certificates.)

The same comparison gives a useful finite-range consequence.  To the best of
our knowledge, after combining the previously published results with
Theorem~\ref{thm:main} and Proposition~\ref{prop:computational}, there is no
remaining unresolved triple $r\le s\le t$ satisfying the necessary
conditions and
\[
|E(K_{r,s,t})|=rs+rt+st<4400.
\]
Thus the boundary construction and the 117 computational certificates cover
all cases below this edge threshold that were unresolved in the prior
literature, as far as we have been able to determine.

The table also gives the forced singleton-type counts from Equation~(5).  One
listed instance lies on the boundary and is therefore also covered by
Theorem~\ref{thm:main}; its computer-generated decomposition provides an
independent explicit check.

\begin{proposition}[Computational certificates]\label{prop:computational}
For every triple $(r,s,t)$ in Table~\ref{tab:computational-cases},
$K_{r,s,t}$ admits a $C_5$-decomposition.
\end{proposition}

\begin{proof}
For each row there is a machine-readable cycle-list certificate in the Zenodo
archive.  Running the verifier of Section~\ref{subsec:verification} on that
file checks that its cycles partition the entire edge set of the stated
complete tripartite graph.
\end{proof}

\begingroup
\small
\setlength{\tabcolsep}{4.2pt}
\renewcommand{\arraystretch}{1.08}
\begin{longtable}{@{}l r l @{\hspace{1.25em}} l r l@{}}
\caption{Complete tripartite graphs for which an explicit computational $C_5$-decomposition certificate was obtained in this work. For each graph, $|E|=rs+rt+st$, and $(x_A,x_B,x_C)$ gives the forced numbers of cycles whose singleton part is $A$, $B$, and $C$, respectively. The number of cycles is $|E|/5=x_A+x_B+x_C$.}
\label{tab:computational-cases}\\
\toprule
Graph & $|E|$ & $(x_A,x_B,x_C)$ &
Graph & $|E|$ & $(x_A,x_B,x_C)$\\
\midrule
\endfirsthead
\multicolumn{6}{c}{\tablename~\thetable\ (continued)}\\
\toprule
Graph & $|E|$ & $(x_A,x_B,x_C)$ &
Graph & $|E|$ & $(x_A,x_B,x_C)$\\
\midrule
\endhead
\midrule
\multicolumn{6}{r}{Continued on next page}\\
\endfoot
\bottomrule
\endlastfoot
$K_{9,19,23}$ & 815 & $(137,22,4)$ &
$K_{11,17,21}$ & 775 & $(101,38,16)$ \\
$K_{11,21,27}$ & 1{,}095 & $(174,39,6)$ & $K_{11,27,31}$ & 1{,}475 & $(271,23,1)$ \\
$K_{13,19,29}$ & 1{,}175 & $(158,71,6)$ & $K_{13,21,23}$ & 1{,}055 & $(136,44,31)$ \\
$K_{13,23,31}$ & 1{,}415 & $(215,60,8)$ & $K_{13,31,33}$ & 1{,}855 & $(326,29,16)$ \\
$K_{15,21,25}$ & 1{,}215 & $(141,66,36)$ & 
$K_{15,21,35}$ & 1{,}575 & $(210,105,0)$ \\
$K_{15,23,25}$ & 1{,}295 & $(158,58,43)$ & 
$K_{15,23,35}$ & 1{,}675 & $(235,95,5)$ \\
$K_{15,25,27}$ & 1{,}455 & $(192,57,42)$ & $K_{15,25,29}$ & 1{,}535 & $(209,64,34)$ \\
$K_{15,25,31}$ & 1{,}615 & $(226,71,26)$ & $K_{15,25,33}$ & 1{,}695 & $(243,78,18)$ \\
$K_{15,25,35}$ & 1{,}775 & $(260,85,10)$ & $K_{15,25,37}$ & 1{,}855 & $(277,92,2)$ \\
$K_{15,27,35}$ & 1{,}875 & $(285,75,15)$ & $K_{15,29,35}$ & 1{,}975 & $(310,65,20)$ \\
$K_{15,31,35}$ & 2{,}075 & $(335,55,25)$ & $K_{15,33,35}$ & 2{,}175 & $(360,45,30)$ \\
$K_{15,35,37}$ & 2{,}375 & $(410,40,25)$ & $K_{15,35,39}$ & 2{,}475 & $(435,45,15)$ \\
$K_{15,35,41}$ & 2{,}575 & $(460,50,5)$ & $K_{17,19,27}$ & 1{,}295 & $(127,100,32)$ \\
$K_{17,21,31}$ & 1{,}535 & $(172,110,25)$ & $K_{17,25,35}$ & 1{,}895 & $(248,108,23)$ \\
$K_{17,27,29}$ & 1{,}735 & $(218,73,56)$ & $K_{17,27,39}$ & 2{,}175 & $(309,114,12)$ \\
$K_{17,29,37}$ & 2{,}195 & $(317,95,27)$ & $K_{17,31,41}$ & 2{,}495 & $(386,99,14)$ \\
$K_{17,35,45}$ & 2{,}935 & $(494,89,4)$ & $K_{17,37,39}$ & 2{,}735 & $(448,58,41)$ \\
$K_{17,39,47}$ & 3{,}295 & $(587,70,2)$ & $K_{17,47,49}$ & 3{,}935 & $(758,23,6)$ \\
$K_{19,23,29}$ & 1{,}655 & $(168,110,53)$ & $K_{19,23,39}$ & 2{,}075 & $(241,163,11)$ \\
$K_{19,25,35}$ & 2{,}015 & $(236,131,36)$ & $K_{19,27,37}$ & 2{,}215 & $(278,130,35)$ \\
$K_{19,29,33}$ & 2{,}135 & $(265,100,62)$ & $K_{19,29,43}$ & 2{,}615 & $(362,147,14)$ \\
$K_{19,33,39}$ & 2{,}655 & $(378,105,48)$ & $K_{19,35,45}$ & 3{,}095 & $(478,118,23)$ \\
$K_{19,37,47}$ & 3{,}335 & $(536,113,18)$ & $K_{19,39,43}$ & 3{,}235 & $(515,85,47)$ \\
$K_{19,43,49}$ & 3{,}855 & $(668,80,23)$ & $K_{21,23,33}$ & 1{,}935 & $(186,153,48)$ \\
$K_{21,23,43}$ & 2{,}375 & $(257,214,4)$ & $K_{21,25,35}$ & 2{,}135 & $(224,154,49)$ \\
$K_{21,25,45}$ & 2{,}595 & $(303,213,3)$ & $K_{21,27,31}$ & 2{,}055 & $(213,120,78)$ \\
$K_{21,27,41}$ & 2{,}535 & $(300,177,30)$ & $K_{21,31,37}$ & 2{,}575 & $(316,131,68)$ \\
$K_{21,31,47}$ & 3{,}095 & $(419,184,16)$ & $K_{21,33,43}$ & 3{,}015 & $(408,150,45)$ \\
$K_{21,35,45}$ & 3{,}255 & $(462,147,42)$ & $K_{21,37,41}$ & 3{,}155 & $(443,115,73)$ \\
$K_{21,37,51}$ & 3{,}735 & $(570,162,15)$ & $K_{21,41,47}$ & 3{,}775 & $(586,116,53)$ \\
$K_{21,43,53}$ & 4{,}295 & $(710,127,22)$ & $K_{23,25,35}$ & 2{,}255 & $(212,177,62)$ \\
$K_{23,25,45}$ & 2{,}735 & $(289,244,14)$ & $K_{23,29,39}$ & 2{,}695 & $(296,179,64)$ \\
$K_{23,29,49}$ & 3{,}215 & $(389,242,12)$ & $K_{23,31,33}$ & 2{,}495 & $(262,130,107)$ \\
$K_{23,31,43}$ & 3{,}035 & $(363,191,53)$ & $K_{23,33,41}$ & 3{,}055 & $(371,166,74)$ \\
$K_{23,33,51}$ & 3{,}615 & $(480,225,18)$ & $K_{23,35,45}$ & 3{,}415 & $(446,176,61)$ \\
$K_{23,35,55}$ & 3{,}995 & $(563,233,3)$ & $K_{23,39,49}$ & 3{,}935 & $(562,170,55)$ \\
$K_{23,41,43}$ & 3{,}695 & $(512,125,102)$ & $K_{23,41,53}$ & 4{,}335 & $(653,176,38)$ \\
$K_{23,43,51}$ & 4{,}355 & $(661,151,59)$ & $K_{25,27,35}$ & 2{,}495 & $(223,188,88)$ \\
$K_{25,27,45}$ & 3{,}015 & $(306,261,36)$ & $K_{25,29,35}$ & 2{,}615 & $(246,176,101)$ \\
$K_{25,29,45}$ & 3{,}155 & $(337,247,47)$ & $K_{25,31,35}$ & 2{,}735 & $(269,164,114)$ \\
$K_{25,31,45}$ & 3{,}295 & $(368,233,58)$ & $K_{25,31,55}$ & 3{,}855 & $(467,302,2)$ \\
$K_{25,33,35}$ & 2{,}855 & $(292,152,127)$ & $K_{25,33,45}$ & 3{,}435 & $(399,219,69)$ \\
$K_{25,33,55}$ & 4{,}015 & $(506,286,11)$ & $K_{25,35,37}$ & 3{,}095 & $(338,153,128)$ \\
$K_{25,35,39}$ & 3{,}215 & $(361,166,116)$ & $K_{25,35,41}$ & 3{,}335 & $(384,179,104)$ \\
$K_{25,35,43}$ & 3{,}455 & $(407,192,92)$ & $K_{25,35,45}$ & 3{,}575 & $(430,205,80)$ \\
$K_{25,35,47}$ & 3{,}695 & $(453,218,68)$ & $K_{25,35,49}$ & 3{,}815 & $(476,231,56)$ \\
$K_{25,35,51}$ & 3{,}935 & $(499,244,44)$ & $K_{25,35,53}$ & 4{,}055 & $(522,257,32)$ \\
$K_{25,35,55}$ & 4{,}175 & $(545,270,20)$ & $K_{25,35,57}$ & 4{,}295 & $(568,283,8)$ \\
$K_{25,37,45}$ & 3{,}715 & $(461,191,91)$ & $K_{25,37,55}$ & 4{,}335 & $(584,254,29)$ \\
$K_{25,39,45}$ & 3{,}855 & $(492,177,102)$ & $K_{25,41,45}$ & 3{,}995 & $(523,163,113)$ \\
$K_{25,43,45}$ & 4{,}135 & $(554,149,124)$ & $K_{27,29,37}$ & 2{,}855 & $(251,214,106)$ \\
$K_{27,29,47}$ & 3{,}415 & $(340,293,50)$ & $K_{27,31,41}$ & 3{,}215 & $(314,232,97)$ \\
$K_{27,31,51}$ & 3{,}795 & $(411,309,39)$ & $K_{27,35,45}$ & 3{,}735 & $(414,234,99)$ \\
$K_{27,35,55}$ & 4{,}355 & $(527,307,37)$ & $K_{27,37,49}$ & 4{,}135 & $(493,248,86)$ \\ $K_{29,33,39}$ & 3{,}375 & $(306,228,141)$ & $K_{29,33,49}$ & 3{,}995 & $(409,311,79)$ \\
$K_{29,35,45}$ & 3{,}895 & $(398,263,118)$ & $K_{29,37,47}$ & 4{,}175 & $(452,264,119)$ \\
$K_{29,39,43}$ & 4{,}055 & $(433,218,160)$ & $K_{31,33,43}$ & 3{,}775 & $(332,289,134)$ \\
$K_{31,35,45}$ & 4{,}055 & $(382,292,137)$ & $K_{31,37,41}$ & 3{,}935 & $(365,242,180)$ \\
$K_{33,35,45}$ & 4{,}215 & $(366,321,156)$ &  &  &  \\
\end{longtable}
\endgroup

\subsection{A human-readable matrix rendering}
\label{subsec:matrix-format}

For publication and manual inspection it is useful to render a cycle-list
certificate as three edge-label matrices.  For a decomposition with numbered
cycles $1,\ldots,N$, define
\[
M_{AB}\in\{1,\ldots,N\}^{r\times s},\quad
M_{AC}\in\{1,\ldots,N\}^{r\times t},\quad
M_{BC}\in\{1,\ldots,N\}^{s\times t},
\]
where $M_{XY}(i,j)=k$ means that the edge $X_iY_j$ belongs to cycle $k$.
There is one matrix cell for every graph edge.  Every cycle label occurs five
times across the three matrices, and its five cells are exactly the five
edges of that cycle.  If the singleton part of the cycle is $A$, $B$, or $C$,
then its multiplicities in $(M_{AB},M_{AC},M_{BC})$ are respectively
$(1,1,3)$, $(1,3,1)$, or $(3,1,1)$.

Appendix~\ref{app:k9matrix} gives the complete edge-label matrices for
$K_{9,19,23}$, and Appendix~\ref{app:k9cycles} gives the corresponding
canonical numbered cycle list.  This example illustrates both the compact
machine-readable certificate and a format suitable for human auditing.

\section{Remarks}

\begin{enumerate}[label=\arabic*.]
\item
The proof is entirely constructive. Given an odd boundary triple, the
decomposition can be written down directly using arithmetic in three cyclic
groups.

\item
The role of the boundary condition is transparent. Equation~(5) shows that
equality in the Mahmoodian--Mirzakhani inequality forces one of the three
possible cycle types to disappear completely. The identities
$3u+v=a$ and $u+3v=b$
express exactly how the remaining two cycle types distribute the $AC$- and
$BC$-difference classes.

\item
The parameterization
$(ga,gb,ab)$ where $a+b=4g$
shows that the pairwise-product structure visible in examples such as
$(15,21,35)$,
$(35,65,91)$, and
$(45,55,99)$
is not accidental but is forced by the arithmetic of primitive odd boundary
triples. The endpoint $b=3a$ reduces to the $K_{m,3m,3m}$ family already
constructed by Mahmoodian and Mirzakhani.

\item
The boundary construction leaves open the general strict-interior problem
\[
t<\frac{4rs}{r+s}.
\]
Section~\ref{sec:computation} settles 117 individual cases by explicit
certificate.  Together with the boundary theorem and the earlier literature,
this leaves, to the best of our knowledge, no unresolved admissible case with
fewer than $4400$ edges.  This finite-range completion does not replace a
general construction.  In the strict
interior all three cycle types are generally present, so the two complementary
difference blocks used on the boundary must be replaced by a genuinely three-way
structure.
\end{enumerate}

\section*{Data availability}

The paper, all machine-readable decomposition certificates, and the independent verifier
\texttt{verify\_\allowbreak c5\_\allowbreak tripartite.py} have
been archived together in a versioned Zenodo record.  The DOI of the archived
version corresponding to this manuscript is \ZenodoDOI. The Zenodo record is the archival source
for the certificates used in Proposition~\ref{prop:computational}.

\section*{Acknowledgment}
Sharareh Alipour kindly reviewed an earlier draft of this paper.

\section*{Disclosure}

GPT-5.6 Sol (OpenAI) and Claude Opus 5 (Anthropic) were used in the development of this work to assist with proof development and verification, computer code and computations, and drafting and restructuring portions of the manuscript. The author reviewed and independently verified the resulting mathematical arguments, computations, and text and takes full responsibility for the contents of the paper.

\appendix

\begin{landscape}
\section{Edge-label matrix certificate for $K_{9,19,23}$}
\label{app:k9matrix}

For this graph, $|E|=815$, so a decomposition contains $163$ cycles.  We
canonicalize each cycle up to rotation and reversal, sort the canonical cycles
lexicographically, and label them $001,\ldots,163$.  The resulting singleton
counts are
\[
(x_A,x_B,x_C)=(137,22,4).
\]
The three matrices below contain exactly one entry for every cross-part edge.
For example, label $001$ corresponds to the cycle
\[
A_0B_0A_1B_1C_0A_0,
\]
and appears in the five matrix cells representing its five edges.
\\[1em]
\begingroup
\footnotesize
\setlength{\tabcolsep}{2.15pt}
\renewcommand{\arraystretch}{1.03}
\noindent\textbf{$M_{AB}$}\par\vspace{2mm}
\begin{center}
\begin{tabular}{r|rrrrrrrrrrrrrrrrrrr}
 & \texttt{B0} & \texttt{B1} & \texttt{B2} & \texttt{B3} & \texttt{B4} & \texttt{B5} & \texttt{B6} & \texttt{B7} & \texttt{B8} & \texttt{B9} & \texttt{B10} & \texttt{B11} & \texttt{B12} & \texttt{B13} & \texttt{B14} & \texttt{B15} & \texttt{B16} & \texttt{B17} & \texttt{B18} \\ \hline
\texttt{A0} & 001 & 002 & 003 & 004 & 005 & 006 & 007 & 008 & 009 & 010 & 011 & 012 & 013 & 014 & 015 & 016 & 017 & 018 & 019 \\
\texttt{A1} & 001 & 001 & 022 & 023 & 024 & 025 & 026 & 027 & 028 & 029 & 030 & 031 & 032 & 033 & 034 & 035 & 036 & 028 & 037 \\
\texttt{A2} & 042 & 043 & 044 & 045 & 046 & 047 & 048 & 049 & 050 & 051 & 050 & 052 & 053 & 054 & 055 & 056 & 057 & 044 & 058 \\
\texttt{A3} & 062 & 063 & 044 & 064 & 065 & 066 & 067 & 068 & 038 & 069 & 070 & 071 & 072 & 073 & 074 & 075 & 076 & 077 & 060 \\
\texttt{A4} & 079 & 080 & 081 & 082 & 083 & 084 & 085 & 021 & 061 & 086 & 087 & 088 & 020 & 089 & 090 & 091 & 092 & 093 & 094 \\
\texttt{A5} & 097 & 098 & 099 & 100 & 101 & 102 & 103 & 104 & 105 & 106 & 107 & 108 & 095 & 109 & 110 & 111 & 112 & 113 & 114 \\
\texttt{A6} & 115 & 116 & 117 & 118 & 041 & 119 & 059 & 120 & 121 & 122 & 050 & 123 & 124 & 125 & 078 & 126 & 127 & 028 & 128 \\
\texttt{A7} & 130 & 131 & 132 & 133 & 134 & 135 & 136 & 096 & 137 & 138 & 129 & 139 & 140 & 141 & 142 & 039 & 143 & 040 & 144 \\
\texttt{A8} & 145 & 146 & 147 & 148 & 149 & 150 & 151 & 152 & 153 & 154 & 155 & 156 & 157 & 158 & 159 & 160 & 161 & 162 & 163 \\
\end{tabular}
\end{center}
\vspace{3mm}
\noindent\textbf{$M_{AC}$}\par\vspace{2mm}
\begin{center}
\begin{tabular}{r|rrrrrrrrrrrrrrrrrrrrrrr}
 & \texttt{C0} & \texttt{C1} & \texttt{C2} & \texttt{C3} & \texttt{C4} & \texttt{C5} & \texttt{C6} & \texttt{C7} & \texttt{C8} & \texttt{C9} & \texttt{C10} & \texttt{C11} & \texttt{C12} & \texttt{C13} & \texttt{C14} & \texttt{C15} & \texttt{C16} & \texttt{C17} & \texttt{C18} & \texttt{C19} & \texttt{C20} & \texttt{C21} & \texttt{C22} \\ \hline
\texttt{A0} & 001 & 010 & 020 & 016 & 004 & 006 & 005 & 009 & 019 & 012 & 017 & 007 & 014 & 013 & 008 & 003 & 015 & 021 & 021 & 018 & 011 & 002 & 020 \\
\texttt{A1} & 023 & 038 & 034 & 039 & 033 & 024 & 031 & 027 & 037 & 022 & 035 & 040 & 026 & 030 & 025 & 041 & 039 & 040 & 038 & 032 & 029 & 041 & 036 \\
\texttt{A2} & 043 & 037 & 057 & 055 & 059 & 042 & 052 & 054 & 037 & 056 & 060 & 047 & 059 & 058 & 051 & 053 & 046 & 060 & 049 & 061 & 045 & 048 & 061 \\
\texttt{A3} & 063 & 074 & 078 & 066 & 069 & 065 & 062 & 068 & 044 & 077 & 060 & 070 & 067 & 078 & 071 & 053 & 073 & 075 & 038 & 053 & 064 & 072 & 076 \\
\texttt{A4} & 080 & 094 & 020 & 093 & 089 & 090 & 091 & 088 & 083 & 079 & 081 & 095 & 085 & 092 & 087 & 096 & 084 & 021 & 082 & 095 & 086 & 096 & 061 \\
\texttt{A5} & 098 & 104 & 111 & 097 & 103 & 101 & 107 & 112 & 110 & 105 & 113 & 076 & 102 & 058 & 108 & 100 & 099 & 114 & 058 & 095 & 109 & 106 & 076 \\
\texttt{A6} & 116 & 121 & 078 & 113 & 059 & 127 & 123 & 120 & 050 & 126 & 113 & 028 & 122 & 129 & 125 & 115 & 128 & 129 & 117 & 124 & 118 & 041 & 119 \\
\texttt{A7} & 131 & 121 & 137 & 105 & 130 & 134 & 141 & 138 & 133 & 105 & 132 & 040 & 135 & 142 & 144 & 140 & 039 & 129 & 121 & 136 & 139 & 096 & 143 \\
\texttt{A8} & 146 & 136 & 145 & 153 & 149 & 148 & 155 & 157 & 154 & 162 & 007 & 007 & 161 & 147 & 151 & 159 & 156 & 150 & 160 & 136 & 163 & 158 & 152 \\
\end{tabular}
\end{center}
\endgroup
\newpage
\begingroup
\footnotesize
\setlength{\tabcolsep}{2.15pt}
\renewcommand{\arraystretch}{1.03}
\noindent\textbf{$M_{BC}$}\par\vspace{2mm}
\begin{center}
\begin{tabular}{r|rrrrrrrrrrrrrrrrrrrrrrr}
 & \texttt{C0} & \texttt{C1} & \texttt{C2} & \texttt{C3} & \texttt{C4} & \texttt{C5} & \texttt{C6} & \texttt{C7} & \texttt{C8} & \texttt{C9} & \texttt{C10} & \texttt{C11} & \texttt{C12} & \texttt{C13} & \texttt{C14} & \texttt{C15} & \texttt{C16} & \texttt{C17} & \texttt{C18} & \texttt{C19} & \texttt{C20} & \texttt{C21} & \texttt{C22} \\ \hline
\texttt{B0} & 146 & 002 & 100 & 004 & 004 & 148 & 148 & 146 & 133 & 064 & 079 & 133 & 062 & 097 & 042 & 100 & 082 & 130 & 082 & 115 & 064 & 002 & 145 \\
\texttt{B1} & 001 & 002 & 145 & 097 & 130 & 042 & 062 & 146 & 063 & 079 & 079 & 043 & 062 & 097 & 042 & 115 & 098 & 130 & 116 & 115 & 080 & 131 & 145 \\
\texttt{B2} & 080 & 022 & 075 & 055 & 161 & 081 & 005 & 140 & 132 & 126 & 126 & 070 & 161 & 117 & 005 & 140 & 055 & 075 & 070 & 147 & 080 & 003 & 099 \\
\texttt{B3} & 116 & 010 & 100 & 004 & 045 & 118 & 148 & 138 & 154 & 064 & 138 & 133 & 122 & 122 & 051 & 154 & 082 & 051 & 116 & 010 & 023 & 106 & 106 \\
\texttt{B4} & 043 & 104 & 083 & 093 & 033 & 081 & 085 & 134 & 024 & 046 & 081 & 043 & 085 & 104 & 005 & 041 & 065 & 049 & 049 & 093 & 149 & 033 & 101 \\
\texttt{B5} & 098 & 052 & 066 & 153 & 149 & 025 & 052 & 088 & 047 & 012 & 150 & 012 & 153 & 119 & 102 & 135 & 098 & 084 & 048 & 006 & 149 & 048 & 088 \\
\texttt{B6} & 131 & 136 & 066 & 066 & 157 & 006 & 085 & 157 & 047 & 026 & 007 & 047 & 059 & 119 & 067 & 151 & 084 & 084 & 048 & 006 & 103 & 131 & 119 \\
\texttt{B7} & 063 & 068 & 057 & 124 & 103 & 152 & 008 & 009 & 063 & 046 & 057 & 027 & 067 & 104 & 067 & 096 & 046 & 049 & 021 & 124 & 103 & 009 & 120 \\
\texttt{B8} & 023 & 038 & 142 & 105 & 159 & 065 & 008 & 112 & 050 & 077 & 077 & 028 & 153 & 142 & 008 & 159 & 065 & 137 & 121 & 061 & 023 & 009 & 112 \\
\texttt{B9} & 029 & 125 & 054 & 014 & 086 & 118 & 141 & 054 & 069 & 158 & 138 & 073 & 014 & 122 & 125 & 154 & 073 & 051 & 141 & 010 & 118 & 158 & 106 \\
\texttt{B10} & 011 & 074 & 137 & 128 & 086 & 032 & 087 & 120 & 132 & 074 & 132 & 030 & 155 & 129 & 108 & 107 & 128 & 137 & 070 & 032 & 086 & 108 & 120 \\
\texttt{B11} & 139 & 052 & 017 & 016 & 045 & 025 & 071 & 156 & 019 & 031 & 017 & 012 & 160 & 013 & 025 & 123 & 016 & 013 & 160 & 019 & 045 & 108 & 088 \\
\texttt{B12} & 139 & 094 & 083 & 124 & 157 & 032 & 087 & 140 & 083 & 026 & 094 & 095 & 026 & 092 & 087 & 072 & 015 & 013 & 092 & 053 & 139 & 015 & 020 \\
\texttt{B13} & 109 & 125 & 054 & 014 & 069 & 152 & 155 & 156 & 069 & 158 & 150 & 073 & 155 & 147 & 151 & 151 & 156 & 150 & 141 & 147 & 089 & 033 & 152 \\
\texttt{B14} & 011 & 110 & 142 & 036 & 159 & 024 & 091 & 027 & 024 & 074 & 035 & 027 & 135 & 078 & 091 & 135 & 055 & 035 & 034 & 090 & 011 & 015 & 036 \\
\texttt{B15} & 029 & 068 & 075 & 039 & 162 & 101 & 123 & 068 & 111 & 162 & 126 & 056 & 160 & 117 & 091 & 123 & 016 & 035 & 117 & 018 & 029 & 018 & 101 \\
\texttt{B16} & 163 & 110 & 017 & 036 & 161 & 090 & 031 & 127 & 110 & 031 & 057 & 076 & 102 & 114 & 102 & 072 & 143 & 114 & 092 & 090 & 163 & 072 & 112 \\
\texttt{B17} & 109 & 022 & 034 & 113 & 162 & 127 & 107 & 127 & 044 & 022 & 077 & 030 & 144 & 030 & 144 & 107 & 099 & 040 & 034 & 093 & 109 & 018 & 099 \\
\texttt{B18} & 163 & 037 & 111 & 128 & 089 & 134 & 071 & 134 & 111 & 056 & 094 & 056 & 144 & 114 & 071 & 003 & 143 & 060 & 058 & 019 & 089 & 003 & 143 \\
\end{tabular}
\end{center}
\vspace{2mm}

\noindent\textit{Matrix verification rule.} Every matrix cell represents one
edge.  Each label occurs exactly five times, and the five cells carrying that
label are the edges of the corresponding cycle in
Appendix~\ref{app:k9cycles}.  Hence the matrices give a complete human-readable
rendering of the cycle-list certificate.
\endgroup
\end{landscape}

\section{Canonical cycle-list certificate for $K_{9,19,23}$}
\label{app:k9cycles}

The complete numbered certificate corresponding to the matrices in
Appendix~\ref{app:k9matrix} is listed below.  Each line gives one 5-cycle; the
closing edge from the fifth vertex back to the first is implicit.

\begingroup
\scriptsize\ttfamily
\setlength{\columnsep}{1.2em}
\begin{multicols}{3}
\noindent
001 A0 B0 A1 B1 C0\\
002 A0 B1 C1 B0 C21\\
003 A0 B2 C21 B18 C15\\
004 A0 B3 C3 B0 C4\\
005 A0 B4 C14 B2 C6\\
006 A0 B5 C19 B6 C5\\
007 A0 B6 C10 A8 C11\\
008 A0 B7 C6 B8 C14\\
009 A0 B8 C21 B7 C7\\
010 A0 B9 C19 B3 C1\\
011 A0 B10 C0 B14 C20\\
012 A0 B11 C11 B5 C9\\
013 A0 B12 C17 B11 C13\\
014 A0 B13 C3 B9 C12\\
015 A0 B14 C21 B12 C16\\
016 A0 B15 C16 B11 C3\\
017 A0 B16 C2 B11 C10\\
018 A0 B17 C21 B15 C19\\
019 A0 B18 C19 B11 C8\\
020 A0 C2 A4 B12 C22\\
021 A0 C17 A4 B7 C18\\
022 A1 B2 C1 B17 C9\\
023 A1 B3 C20 B8 C0\\
024 A1 B4 C8 B14 C5\\
025 A1 B5 C5 B11 C14\\
026 A1 B6 C9 B12 C12\\
027 A1 B7 C11 B14 C7\\
028 A1 B8 C11 A6 B17\\
029 A1 B9 C0 B15 C20\\
030 A1 B10 C11 B17 C13\\
031 A1 B11 C9 B16 C6\\
032 A1 B12 C5 B10 C19\\
033 A1 B13 C21 B4 C4\\
034 A1 B14 C18 B17 C2\\
035 A1 B15 C17 B14 C10\\
036 A1 B16 C3 B14 C22\\
037 A1 B18 C1 A2 C8\\
038 A1 C1 B8 A3 C18\\
039 A1 C3 B15 A7 C16\\
040 A1 C11 A7 B17 C17\\
041 A1 C15 B4 A6 C21\\
042 A2 B0 C14 B1 C5\\
043 A2 B1 C11 B4 C0\\
044 A2 B2 A3 C8 B17\\
045 A2 B3 C4 B11 C20\\
046 A2 B4 C9 B7 C16\\
047 A2 B5 C8 B6 C11\\
048 A2 B6 C18 B5 C21\\
049 A2 B7 C17 B4 C18\\
050 A2 B8 C8 A6 B10\\
051 A2 B9 C17 B3 C14\\
052 A2 B11 C1 B5 C6\\
053 A2 B12 C19 A3 C15\\
054 A2 B13 C2 B9 C7\\
055 A2 B14 C16 B2 C3\\
056 A2 B15 C11 B18 C9\\
057 A2 B16 C10 B7 C2\\
058 A2 B18 C18 A5 C13\\
059 A2 C4 A6 B6 C12\\
060 A2 C10 A3 B18 C17\\
061 A2 C19 B8 A4 C22\\
062 A3 B0 C12 B1 C6\\
063 A3 B1 C8 B7 C0\\
064 A3 B3 C9 B0 C20\\
065 A3 B4 C16 B8 C5\\
066 A3 B5 C2 B6 C3\\
067 A3 B6 C14 B7 C12\\
068 A3 B7 C1 B15 C7\\
069 A3 B9 C8 B13 C4\\
070 A3 B10 C18 B2 C11\\
071 A3 B11 C6 B18 C14\\
072 A3 B12 C15 B16 C21\\
073 A3 B13 C11 B9 C16\\
074 A3 B14 C9 B10 C1\\
075 A3 B15 C2 B2 C17\\
076 A3 B16 C11 A5 C22\\
077 A3 B17 C10 B8 C9\\
078 A3 C2 A6 B14 C13\\
079 A4 B0 C10 B1 C9\\
080 A4 B1 C20 B2 C0\\
081 A4 B2 C5 B4 C10\\
082 A4 B3 C16 B0 C18\\
083 A4 B4 C2 B12 C8\\
084 A4 B5 C17 B6 C16\\
085 A4 B6 C6 B4 C12\\
086 A4 B9 C4 B10 C20\\
087 A4 B10 C6 B12 C14\\
088 A4 B11 C22 B5 C7\\
089 A4 B13 C20 B18 C4\\
090 A4 B14 C19 B16 C5\\
091 A4 B15 C14 B14 C6\\
092 A4 B16 C18 B12 C13\\
093 A4 B17 C19 B4 C3\\
094 A4 B18 C10 B12 C1\\
095 A4 C11 B12 A5 C19\\
096 A4 C15 B7 A7 C21\\
097 A5 B0 C13 B1 C3\\
098 A5 B1 C16 B5 C0\\
099 A5 B2 C22 B17 C16\\
100 A5 B3 C2 B0 C15\\
101 A5 B4 C22 B15 C5\\
102 A5 B5 C14 B16 C12\\
103 A5 B6 C20 B7 C4\\
104 A5 B7 C13 B4 C1\\
105 A5 B8 C3 A7 C9\\
106 A5 B9 C22 B3 C21\\
107 A5 B10 C15 B17 C6\\
108 A5 B11 C21 B10 C14\\
109 A5 B13 C0 B17 C20\\
110 A5 B14 C1 B16 C8\\
111 A5 B15 C8 B18 C2\\
112 A5 B16 C22 B8 C7\\
113 A5 B17 C3 A6 C10\\
114 A5 B18 C13 B16 C17\\
115 A6 B0 C19 B1 C15\\
116 A6 B1 C18 B3 C0\\
117 A6 B2 C13 B15 C18\\
118 A6 B3 C5 B9 C20\\
119 A6 B5 C13 B6 C22\\
120 A6 B7 C22 B10 C7\\
121 A6 B8 C18 A7 C1\\
122 A6 B9 C13 B3 C12\\
123 A6 B11 C15 B15 C6\\
124 A6 B12 C3 B7 C19\\
125 A6 B13 C1 B9 C14\\
126 A6 B15 C10 B2 C9\\
127 A6 B16 C7 B17 C5\\
128 A6 B18 C3 B10 C16\\
129 A6 C13 B10 A7 C17\\
130 A7 B0 C17 B1 C4\\
131 A7 B1 C21 B6 C0\\
132 A7 B2 C8 B10 C10\\
133 A7 B3 C11 B0 C8\\
134 A7 B4 C7 B18 C5\\
135 A7 B5 C15 B14 C12\\
136 A7 B6 C1 A8 C19\\
137 A7 B8 C17 B10 C2\\
138 A7 B9 C10 B3 C7\\
139 A7 B11 C0 B12 C20\\
140 A7 B12 C7 B2 C15\\
141 A7 B13 C18 B9 C6\\
142 A7 B14 C2 B8 C13\\
143 A7 B16 C16 B18 C22\\
144 A7 B18 C12 B17 C14\\
145 A8 B0 C22 B1 C2\\
146 A8 B1 C7 B0 C0\\
147 A8 B2 C19 B13 C13\\
148 A8 B3 C6 B0 C5\\
149 A8 B4 C20 B5 C4\\
150 A8 B5 C10 B13 C17\\
151 A8 B6 C15 B13 C14\\
152 A8 B7 C5 B13 C22\\
153 A8 B8 C12 B5 C3\\
154 A8 B9 C15 B3 C8\\
155 A8 B10 C12 B13 C6\\
156 A8 B11 C7 B13 C16\\
157 A8 B12 C4 B6 C7\\
158 A8 B13 C9 B9 C21\\
159 A8 B14 C4 B8 C15\\
160 A8 B15 C12 B11 C18\\
161 A8 B16 C4 B2 C12\\
162 A8 B17 C4 B15 C9\\
163 A8 B18 C0 B16 C20\\
\end{multicols}
\endgroup

\section{Independent Python verifier}
\label{app:verifier}

The following is the independent verifier used to check the cycle-list
certificates.  It has no dependency on the exact-cover solver.

\begingroup
\small
\begin{verbatim}
#!/usr/bin/env python3
import re
import sys
from pathlib import Path

filename = sys.argv[1]

r, s, t = map(int, re.match(
    r'^k(\d+)_(\d+)_(\d+)',
    Path(filename).name,
    re.I
).groups())

sizes = {"A": r, "B": s, "C": t}
E = r*s + r*t + s*t

edges = set()
cycles = 0

for line in open(filename):
    if not line.strip():
        continue

    c = line.split()
    assert len(c) == 5
    assert len(set(c)) == 5

    for v in c:
        assert v[0] in sizes
        assert 0 <= int(v[1:]) < sizes[v[0]]

    for i in range(5):
        u, v = c[i], c[(i + 1) % 5]
        assert u[0] != v[0]
        edges.add(tuple(sorted((u, v))))

    cycles += 1

assert cycles * 5 == E
assert len(edges) == E

print(f"VERIFIED: K_{{{r},{s},{t}}}, {cycles} cycles, {E} edges")
\end{verbatim}
\endgroup

\end{document}